\documentclass[preprint]{article}
\usepackage{graphicx, hyperref, relsize, float}
\usepackage{pdflscape, multirow, tabularray} 
\usepackage{fullpage, microtype, amsmath, amsthm, enumerate, amsfonts, amssymb, hyperref, xcolor, mathtools, tikz-cd, colonequals}
\usepackage[inner=1.0in,outer=1.0in,bottom=1.0in, top=1.0in]{geometry}

\usepackage[all,arc,curve,frame,color]{xy}
\usepackage{subfigure}

\newcommand{\Q}{\mathbb{Q}}
\newcommand{\Z}{\mathbb{Z}}

\theoremstyle{definition}

\renewcommand{\P}{\mathbb{P}}

\newcommand{\Gal}{\textnormal{Gal}}
\newcommand{\Aut}{\textnormal{Aut}}

\newtheorem{theorem}{Theorem}[section]

\newtheorem{prop}[theorem]{Proposition}

\newtheorem{definition}[theorem]{Definition}
\newtheorem{example}[theorem]{Example}
\newtheorem{remark}[theorem]{Remark}

\begin{document}


\title{Models of Elliptic Curves with Maximal 2-adic Image Subject to a Rational 2-Isogeny}
\author{Roberto Hernandez}

\maketitle

% ...




\begin{abstract}
    \noindent Let $E$ be a non-CM elliptic curve defined over $\mathbb{Q}$. We present models of curves that have 2-adic image as large as possible, given the constraint that they have a rational 2-isogeny. These curves are known to be generic, and one of the aims of this article is to make that fact explicit. Our techniques make use of modular curves and their $j$-maps, as well as classification results on isogeny-graphs that have appeared in the literature recently. 
\end{abstract}

\section{Introduction}
Let $E$ be an elliptic curve defined over $\Q$ without CM. Let $n > 1$ be an integer. We denote by $E[n]$ the $n$-torsion subgroup of $E(\overline{\Q})$. The absolute Galois group $G_\Q \colonequals  \Gal(\overline{\Q}/\Q)$ acts on $E[n]$ by its action on the coordinates of the points. Moreover, since $E[n]$ is a free $\Z/n\Z$ module of rank 2, we may identify (after fixing a basis) $\Aut(E[n])$ with $\text{GL}_2(\Z/n\Z)$. Then we have the following mod $n$ Galois representation attached to $E$,
\begin{align*}
    \rho_{E,n} \colon G_\Q \to \Aut(E[n]) \cong \text{GL}_2(\Z/n\Z).
\end{align*}
Thus, we may view the image $\rho_{E,n}(G_\Q)$ as a subgroup of $\text{GL}_2(\Z/n\Z)$, determined up to conjugacy. Fix a prime $\ell$ and choose compatible bases for $E[\ell^k]$, for all $k$. Doing so, we can take the inverse limit of these mod $\ell^k$ Galois representations and obtain another representation,
\begin{align*}
    \rho_{E,\ell^\infty} \colon G_\Q \to \text{GL}_2(\Z_\ell),
\end{align*}
called the $\ell$-adic Galois representation attached to $E$. 

A major result in this area is due to Serre \cite{serre1972proprietes}, who proved that for all but finitely many primes $p$, the representation $\rho_{E,p}$ is surjective. More recently, Rouse and Zureick-Brown \cite{rzb2adic} classified, up to conjugacy, the 2-adic images that arise from non-CM elliptic curves over $\Q$. Subsequently, Rouse, Sutherland, and Zureick--Brown studied the analogous question for $\ell$-adic images, obtaining explicit classifications in many cases, while other cases remain conjectural or conditional on unresolved questions about rational points on certain modular curves. In the CM case, Lozano-Robledo \cite{lozano2022galois} obtained a corresponding classification, where the possible images are constrained by the structure of the Cartan subgroups and their normalizers. This fits into a broader landscape of computing possible Galois images. For instance, Zywina \cite{zywina2015possible} studied the possible non-surjective mod $\ell$ images for non-CM elliptic curves over $\Q$, substantially narrowing the possibilities that can occur. 

Let $\mathcal{E}$ be a $\Q$-isogeny class of elliptic curves defined over $\Q$. The isogeny graph associated to $\mathcal{E}$ is a graph whose vertices represent elliptic curves in $\mathcal{E}$, and whose edges represent the cyclic $\Q$-isogeny between two curves in $\mathcal{E}$ labeled with the degree of the isogeny. We define the isogeny-torsion graph to be an isogeny-graph with the additional information that each vertex also records the abstract torsion subgroup of the elliptic curve belonging to that vertex. For example, the isogeny graph
\begin{center}
\begin{minipage}{0.2\textwidth}
        \begin{tikzcd}
        E_1 \arrow[r, "2", no head] \arrow[d, "3"', no head] & E_2 \arrow[d, "3", no head] \\
        E_3 \arrow[r, "2"', no head]                         & E_4                         
        \end{tikzcd}

\end{minipage}
\end{center}
will have isogeny-torsion graph labeled as $([6],[6],[2],[2])$ to denote that the torsion subgroups of $E_1$, $E_2$, $E_3$, and $E_4$ are $\Z/6\Z$, $\Z/6\Z$, $\Z/2\Z$, and $\Z/2\Z$, respectively.

A complete classification of isogeny-torsion graphs that appear over $\Q$ is given in \cite{chiloyanalvaro}. Using this classification of isogeny-torsion graphs, Chiloyan \cite[Section 8]{Chiloyan2adic} classified the 2-adic images that appear for isogeny-torsion graphs of elliptic curves without CM defined over $\Q$. Here, 2-adic images are labeled using the Rouse, Sutherland, and Zureick-Brown label $[N.i.g.n]$ from the LMFDB, where $N$ is the level of the subgroup $H \subset \text{GL}_2(\Z_2)$, $i = [\text{GL}_2(\Z_2) : H]$ is its index, $g$ is the genus of the modular curve $X_H$, and $n$ is a tie-breaker. In particular, the label [2.3.0.1] denotes the unique conjugacy class of level 2, index 3, and genus 0 in $\text{GL}_2(\Z_2)$, corresponding to curves whose mod 2 Galois image is contained in a Borel subgroup (i.e., curves possessing a rational 2-isogeny). This is where our motivation for this project began. We summarize results gathered from \cite[Tables 11, 16, and 17]{Chiloyan2adic}, and set the notation for the different $\Q$-isogeny classes we will be concerned with.

\begin{definition}
    An elliptic curve is $\ell$-maximal if its $\ell$-adic image is as large as possible given the constraint of having a rational $\ell$-isogeny.
\end{definition}

\begin{example}
    The elliptic curve $y^2+xy+y=x^3-16x-25$ is 2-maximal, since it has a rational 2-isogeny and it has 2-adic image [2.3.0.1].
\end{example}

\begin{theorem}[Chiloyan] \label{chiloyan}
    Let $E$ be an elliptic curve defined over $\Q$ without complex multiplication that has 2-adic image [2.3.0.1]. Then the isogeny-torsion graph containing $E$ is 2-maximal and is one of:
    \begin{enumerate}
        \item $L_2(2)$ with isogeny-torsion graph $([2],[2])$
        \begin{minipage}{0.2\textwidth}
        \centering
        \begin{tikzcd}
        E_1 \arrow[r, "2", no head] & E_2
        \end{tikzcd}
        \end{minipage}
        \item $R_4(6)$ with isogeny-torsion graph $([2],[2],[2],[2])$
        \begin{minipage}{0.2\textwidth}
        \centering
        \begin{tikzcd}
        E_1 \arrow[r, "2", no head] \arrow[d, "3"', no head] & E_2 \arrow[d, "3", no head] \\
        E_3 \arrow[r, "2"', no head]                         & E_4                         
        \end{tikzcd}
        \end{minipage}
        \item $R_4(6)$ with isogeny-torsion graph $([6],[6],[2],[2])$
        \begin{minipage}{0.2\textwidth}
        \centering
        \begin{tikzcd}
        E_1 \arrow[r, "2", no head] \arrow[d, "3"', no head] & E_2 \arrow[d, "3", no head] \\
        E_3 \arrow[r, "2"', no head]                         & E_4                         
        \end{tikzcd}
        \end{minipage}
        \item $R_4(10)$ with isogeny-torsion graph $([2],[2],[2],[2])$
        \begin{minipage}{0.2\textwidth}
        \centering
        \begin{tikzcd}
        E_1 \arrow[r, "2", no head] \arrow[d, "5"', no head] & E_2 \arrow[d, "5", no head] \\
        E_3 \arrow[r, "2"', no head]                         & E_4                         
        \end{tikzcd}
        \end{minipage}
        \item $R_4(10)$ with isogeny-torsion graph $([10],[10],[2],[2])$
        \begin{minipage}{0.2\textwidth}
        \centering
        \begin{tikzcd}
        E_1 \arrow[r, "2", no head] \arrow[d, "5"', no head] & E_2 \arrow[d, "5", no head] \\
        E_3 \arrow[r, "2"', no head]                         & E_4                         
        \end{tikzcd}
        \end{minipage}
        \item $R_6$ with isogeny-torsion graph $([2],[2],[2],[2],[2],[2])$
        \begin{minipage}{0.3\textwidth}
        \centering
        \begin{tikzcd}
        E_1 \arrow[r, "3", no head] \arrow[d, "2"', no head] & E_3 \arrow[d, "2"', no head] \arrow[r, "3", no head] & E_5 \arrow[d, "2", no head] \\
        E_2 \arrow[r, "3"', no head]                         & E_4 \arrow[r, "3"', no head]                         & E_6                  
        \end{tikzcd}
        \end{minipage}
        \item $R_6$ with isogeny-torsion graph $([6],[6],[6],[6],[2],[2])$
        \begin{minipage}{0.3\textwidth}
        \centering
        \begin{tikzcd}
        E_1 \arrow[r, "3", no head] \arrow[d, "2"', no head] & E_3 \arrow[d, "2"', no head] \arrow[r, "3", no head] & E_5 \arrow[d, "2", no head] \\
        E_2 \arrow[r, "3"', no head]                         & E_4 \arrow[r, "3"', no head]                         & E_6                  
        \end{tikzcd}.
        \end{minipage}
    \end{enumerate}
\end{theorem}
Our goal is to present Weierstrass models for the curves that have maximal 2-adic image (i.e. [2.3.0.1]) given the presence of a rational 2-isogeny. One could in theory also do this for 2-adic images for curves with CM using the classification given in \cite{ChiloyanCM} and for 3-adic images using the classification given in \cite{Rakvi}. We note that there exists prior work in the literature \cite{MayleRakvi} of finding characterizations for elliptic curves with adelic Galois representation as large as possible, given a constraint on the image modulo 2. Although our work focuses on individual 2-adic images, these adelic results are relevant to our setting and provide a broader perspective on the maximality of the Galois images we consider.

\begin{definition}
    Let $\mathcal{E}$ be a $\Q$-isogeny class of elliptic curves. If every curve in $\mathcal{E}$ is $\ell$-maximal, then we say $\mathcal{E}$ is $\ell$-maximal.
\end{definition}

\noindent In other words, we can restate the main focus of this article as finding models of elliptic curves whose isogeny class $\mathcal{E}$ is 2-maximal. Using Theorem \ref{chiloyan}  above, we know which torsion-subgroups can occur, and which $\Q$-isogeny classes are possible, thereby simplifying our analysis. \\

\noindent \textbf{Computational Software and AI Use.} All the computations performed in this article were done using Magma \cite{bosma1997magma}, or SageMath \cite{sagemath}. Some computational code was developed with assistance from Gemini 3.1 Pro. Some of the algebraic calculations appearing in the results of this article were initially obtained with assistance from Gemini 3.1 Pro, and subsequently independently checked by the author. The author also used ChatGPT 5.6 Sol to check grammar and improve the exposition. The author takes full responsibility for the content in this manuscript.

\section{Preliminaries}

In this section we discuss some facts about modular curves that will help guide the philosophy of most of the results in this article. Let $k$ be a field of characteristic 0. For an integer $N \geq 2$ the non-cuspidal $k$-rational points on the modular curve $X_0(N)$ parameterize $\overline{k}$-isomorphism classes of pairs $(E,C)$, where $E$ is an elliptic curve defined over $k$, and $C$ is a $\Gal_k$-stable cyclic subgroup of $E[N](\overline{k})$ of order $N$. Equivalently, $C$ is the kernel of a $k$-rational isogeny $E \to E'$ of degree $N$. Similarly, the non-cuspidal $k$-rational points on the modular curve $X_1(N)$ correspond to $\overline{k}$-isomorphism classes of pairs $(E,P)$, where $E$ is an elliptic curve over $k$, and $P$ is a point of order $N$ defined over $k$. 

The $\Q$-rational points on $X_0(N)$ have been studied extensively in the literature, culminating in a complete classification for all $N$. One of the landmark results in the classification is when $N$ is prime, due to Mazur \cite{Mazur1978}. The rest of the classification was completed in work by Fricke, Kenku, Klein, Kubert, Ligozat, Mazur and Ogg, among others. The results are spread out in the literature, but one can find them all collected in Tables 3 and 4 in \cite{LozanoRobledo}. When the genus of $X_0(N)$ is at least 2, Faltings' theorem implies that $X_0(N)(\Q)$ is finite. When the genus of $X_0(N)$ is 1, finiteness of its rational points depends on its Mordell-Weil rank. In contrast, when the genus of $X_0(N)$ is zero, there are infinitely many $\overline{\Q}$-isomorphism classes of elliptic curves whose $j$-invariants are parameterized by the image of a rational map. Concretely, when the genus is zero, we have $X_0(N)(\Q) \cong \P^1(\Q)$ and the $j$-map $j_N: X_0(N) \to \P^1(\Q)$ gives the $j$-invariant for an elliptic curve with a rational $N$-isogeny. For instance, one can construct an elliptic curve with a rational 2-isogeny by finding the $j$-map for $X_0(2)$, which is 
\begin{align*}
    j_2(t) = \frac{(256-t)^3}{t^2},
\end{align*}
and then construct the elliptic curve
\begin{align*}
    y^2 + xy = x^3 + 36\frac{t^2}{t^3 + 960t^2 + 196608t -
    16777216}x + \frac{t^2}{t^3 + 960t^2 + 196608t - 16777216},
\end{align*}
with $j_2$ as its $j$-invariant. For all but finitely many values of $t \in \Q$, we obtain an elliptic curve defined over $\Q$ with a rational 2-isogeny. The $\Q$-rational points on $X_1(N)$ were completely classified by Mazur \cite{mazur1977modular}. Concretely, he proved that $X_1(N)(\Q)$ has non-cuspidal points if and only if $N \in \{1,2,3,4,5,6,7,8,9,10,12\}$. For our purposes, all the modular curves we will deal with have genus 0, and thus we will have infinitely many $\overline{\Q}$-isomorphism classes of elliptic curves whose $j$-invariants are parameterized by the image of some $j$-map.

\section{Elliptic Curves with a Rational 2-Isogeny}

\noindent In this section, we will give models of curves that are 2-maximal. Due to the classification in \cite{Chiloyan2adic}, we only need to find curves that  belong to the isogeny classes of type $L_2(2)$, $R_4(6)$, $R_4(10)$, or $R_6$. Let $t \in \Q \setminus \{0,64,256,-512\}$. Let $E_t$ be the elliptic curve over $\Q$ defined by the Weierstrass equation
\begin{align}\label{ECwith2-isog}
   E_t \colon y^2 = x^3 - \left( \frac{t - 256}{2(t + 512)} \right)x^2 + \left( \frac{t(t - 256)^2}{16(t - 64)(t + 512)^2} \right)x.
\end{align}
One can clearly see that $E_t$ has a 2-torsion point at $(0,0)$, and so we have that $\Z/2\Z \hookrightarrow E_t(\Q)$. This curve is obtained by constructing the elliptic curve with $j$-invariant equal to the $j$-map for $X_0(2)$. Note that this model looks different from the one we mentioned before, and this is due to the change of variables $(x, y) \longmapsto \left( x + \frac{t}{4(t + 512)}, y + \frac{1}{2}x \right)$, so that the 2-torsion point is at $(0,0)$. 

\begin{remark}
    This curve may attain extra 2-torsion or odd-degree torsion for certain values of $t$.
\end{remark}

\begin{remark}
    One can in theory use the standard notation for an elliptic curve with a rational 2-torsion point at $(0,0)$, namely
    \begin{align*}
        E_{a,b} \colon y^2 = x^3 + ax^2 + bx. 
    \end{align*}
    However, our calculations in Theorem \ref{L_2-2max} are more conveniently carried out using a one-parameter family, since we only need to impose restrictions on a single parameter rather than two. One can find parameterizations of elliptic curves defined over $\Q$ with prescribed torsion subgroup that use more than one parameter, for example, in \cite{Kubert, barrios2022minimal}.
\end{remark}



\noindent We want to find conditions on $t$ such that the curve $E_t$ described by (\ref{ECwith2-isog}) has 2-adic image [2.3.0.1]. To do this, take the $j$-invariant of this curve which is
$$ j(E_t) = \dfrac{(256-t)^3}{t^2}. $$
For any $t \in \Q \setminus \{0,64,256,-512\}$, the 2-adic image of $E_t$ will be contained in a subgroup conjugate to [2.3.0.1]. We need to make sure it does not land in any of the possible smaller subgroups. Checking the \href{https://users.wfu.edu/rouseja/2adic/}{auxiliary website} associated to \cite{rzb2adic}, we find 12 such possible subgroups and their corresponding $j$-maps. We need to find conditions on $t$ that will guarantee no rational solutions to the arising equations. Doing so we arrive at the following result.

\begin{table}[H]
    \centering
    \begin{tabular}{|c|c|}
        \hline Subgroup & $j$-map \\ \hline
        [2.6.0.1] & $\frac{(t^2 + 192)^3}{(t-8)^2(t+8)^2}$ \\ \hline
        [4.6.0.3] & $\frac{256(t^2+3)^3}{t^2+4}$ \\ \hline
        [4.6.0.5] & $-\frac{64(t^2-3)^3}{(t^2+1)^2}$ \\ \hline
        [4.6.0.4] & $-\frac{(t^2-256)^3}{t^4}$ \\ \hline
        [4.6.0.2] & $\frac{(t^2+256)^3}{t^4}$ \\ \hline
        [4.6.0.1] & $\frac{(t^2 + 16t + 16)^3}{t(t+16)}$\\ \hline
        [8.6.0.2] & $\frac{128(t^2+2)^3}{t^4}$\\ \hline
        [8.6.0.4] & $\frac{64(t^2+6)^3}{t^2+8}$\\ \hline
        [8.6.0.1] & $\frac{64(t^2-6)^3}{t^2-8}$\\ \hline
        [8.6.0.6] & $\frac{32(t^2+6)^3}{(t^2-2)^2}$\\ \hline
        [8.6.0.3] & $-\frac{128(t^2-2)^3}{t^4}$\\ \hline
        [8.6.0.5] & $-\frac{32(t^2-6)^3}{(t^2+2)^2}$\\ \hline
    \end{tabular}
    \caption{Maximal subgroups of [2.3.0.1] and their corresponding $j$-maps.}
    \label{subgroups of [2.3.0.1]}
\end{table}

\begin{theorem} \label{L_2-2max}
    Let $E_t$ be an elliptic curve given as in equation (\ref{ECwith2-isog}). If for all $d \in \{ \pm 1, \pm 2\}$ and $f(t) \in \{ t, t-64, t(t-64) \}$, the rational number $d \cdot f(t)$ is not a perfect square in $\Q$, then $E_t$ is 2-maximal.
\end{theorem}

\begin{proof}
    We compare the target $j$-map for [2.3.0.1] with each of the maximal subgroups listed in Table \ref{subgroups of [2.3.0.1]}. For each subgroup, we obtain an equation whose rational solutions correspond to values of $t$ for which the 2-adic image is contained in that maximal subgroup. The aim is to find restrictions on $t$ so that there are no rational solutions to the resulting equations for all 12 maximal subgroups. We let $X$ denote the variable for the ``new" subgroup $j$-map. \\
    \textbf{The case of [2.6.0.1]}. \\
    We are looking to find when the equation 
    \begin{align*}
        \dfrac{(X^2 + 192)^3}{(X^2 - 64)^2} = \dfrac{(256-t)^3}{t^2},
    \end{align*}
    has no rational solutions. We can relate $X$ and $t$ via $X^2 = 64 - t$, which tells us that for a rational solution to exist, we need $X = \pm \sqrt{64 - t}$. This implies that in order to avoid this subgroup, the condition we need to impose is that $64 - t$ is not a square. \\
    \textbf{The case of [4.6.0.3]}. \\
    We are looking to find when the equation 
    \begin{align*}
        \dfrac{256(X^2 + 3)^3}{X^2 + 4} = \dfrac{(256-t)^3}{t^2},
    \end{align*}
    has no rational solutions. We can relate $X$ and $t$ via the equation $t = \frac{256}{X^2 + 4}$. This implies that 
    \begin{equation*}
        X^2 = \frac{256}{t} - 4 = \frac{256-4t}{t} = \frac{4(64-t)}{t}
    \end{equation*} 
    In order to ensure that $X$ is not rational, we simply need $t(64-t)$ to not be a square. \\
    \textbf{The case of [4.6.0.5]}. \\
    We are looking to find when the equation 
    \begin{align*}
        -\dfrac{64(X^2 - 3)^3}{(X^2 + 1)^2} = \dfrac{(256-t)^3}{t^2},
    \end{align*}
    has no rational solutions. We can relate $X$ and $t$ via the equation $t = 64(X^2 + 1)$. We can rewrite this as $X^2 = \frac{t}{64} - 1 = \frac{t -64}{64}$. Thus, in order to prevent $X$ from being rational, all we need is that $t - 64$ is not a perfect square. \\
    \textbf{The case of [4.6.0.4]}. \\
    We are looking to find when the equation 
    \begin{align*}
        -\dfrac{(X^2 - 256)^3}{X^4} = \dfrac{(256-t)^3}{t^2},
    \end{align*}
    has no rational solutions. One can easily see that in this case the equation relating $X$ and $t$ is $X^2 = t$. Hence, to avoid this subgroup, we need to impose that $t$ is not a perfect square. \\
    \textbf{The case of [4.6.0.2]}. \\
    We are looking to find when the equation 
    \begin{align*}
        \dfrac{(X^2 + 256)^3}{X^4} = \dfrac{(256-t)^3}{t^2},
    \end{align*}
    has no rational solutions. Again, one can easily see here that the equation relating $X$ and $t$ is $t = -X^2$. To avoid this subgroup we ask that $-t$ is not a square in $\Q$. \\
    \textbf{The case of [4.6.0.1]}. \\
    We are looking to find when the equation 
    \begin{align*}
        \dfrac{(X^2 + 16X + 16)^3}{X(X + 16)} = \dfrac{(256-t)^3}{t^2},
    \end{align*}
    has no rational solutions. We can relate $X$ and $t$ via the equation $t = -\frac{4096}{X^2 + 16X}$. For this equation we have that the rational solutions for $X$ will satisfy the quadratic equation
    \begin{align*}
        t X^2 + 16t X+ 4096 = 0.
    \end{align*}
    This quadratic will have rational solutions if its discriminant is a perfect square. The discriminant is given by $\Delta = 256(t^2-64t)$, and so to ensure that $X$ is not rational we impose that $t^2-64t$ is not a perfect square. \\
    \textbf{The case of [8.6.0.2]}. \\
    We are looking to find when the equation 
    \begin{align*}
        \dfrac{128(X^2 + 2)^3}{X^4} = \dfrac{(256-t)^3}{t^2},
    \end{align*}
    has no rational solutions. We can relate $X$ and $t$ via the relation $t = -128X^2$. This implies that $X^2 = -\frac{t}{128}$, and $X$ will be rational if $-2t$ is a perfect square. \\
    \textbf{The case of [8.6.0.4]}. \\
    We are looking to find when the equation 
    \begin{align*}
        \dfrac{64(X^2 + 6)^3}{X^2 + 8} = \dfrac{(256-t)^3}{t^2},
    \end{align*}
    has no rational solutions. We can relate $X$ and $t$ via the equation $t = \frac{512}{X^2+8}$. Solving for $X^2$ we obtain $X^2 = \frac{512}{t}-8 = \frac{512-8t}{t} = \frac{8(64-t)}{t}$. Thus, to ensure that $X$ is not rational we need to impose that $2t(64-t)$ is not a perfect square. \\
    \textbf{The case of [8.6.0.1]}. \\
    We are looking to find when the equation 
    \begin{align*}
        \dfrac{64(X^2 - 6)^3}{X^2 - 8} = \dfrac{(256-t)^3}{t^2},
    \end{align*}
    has no rational solutions. The equation relating $X$ and $t$ is $t = -\frac{512}{X^2-8}$. Solving for $X^2$ we obtain $X^2 = - \frac{512}{t} - 8 = - \frac{512 - 8t}{t} = -\frac{8(64-t)}{t}$. Hence, to avoid rational solutions we require that $-2t(64-t)$ is not a perfect square. \\
    \textbf{The case of [8.6.0.6]}. \\
    We are looking to find when the equation 
    \begin{align*}
        \dfrac{32(X^2 + 6)^3}{(X^2 - 2)^2} = \dfrac{(256-t)^3}{t^2},
    \end{align*}
    has no rational solutions. We relate $X$ and $t$ using the equation $t = -32(X^2-2)$. Solving for $X^2$ gives $X^2 = -\frac{t}{32} + 2 = \frac{(64-t)}{32}$, which implies that to avoid rational solutions we need to impose that $2(64-t)$ is not a perfect square. \\
    \textbf{The case of [8.6.0.3]}. \\
    We are looking to find when the equation 
    \begin{align*}
        \dfrac{-128(X^2 - 2)^3}{X^4} = \dfrac{(256-t)^3}{t^2},
    \end{align*}
    has no rational solutions. The relationship between $X$ and $t$ in this case is $t = 128X^2$. Thus, we have that $X^2 = \frac{t}{128}$, which means that to avoid having rational solutions we need that $2t$ is not a perfect square. \\
    \textbf{The case of [8.6.0.5]}. \\
    We are looking to find when the equation 
    \begin{align*}
        \dfrac{-32(X^2 - 6)^3}{(X^2 + 2)^2} = \dfrac{(256-t)^3}{t^2},
    \end{align*}
    has no rational solutions. The equation relating $X$ and $t$ is given by $t = 32(X^2+2)$. Similarly as before, this implies that $X^2 = \frac{t}{32} - 2 = \frac{t-64}{32}$, and so to prevent rational solutions from existing here we need that $2(t-64)$ is not a perfect square.
\end{proof}

\begin{example}
    Let $t = \frac{64}{103}$. One can verify that $t$ satisfies the conditions stated in Theorem \ref{L_2-2max} and so we define the elliptic curve
    \begin{align*}
        E_\frac{64}{103} \colon y^2 = x^3 + \frac{137}{550}x^2 - \frac{1876}{123420000}x.
    \end{align*}
    This curve is 2-maximal and its rational torsion subgroup is isomorphic to $\Z/2\Z$. Moreover, this curve belongs to an isogeny class of type $L_2(2)$.
\end{example}

While this theorem gives sufficient conditions on the parameter $t$ required for $E_t$ to be 2-maximal, it does not fully control the size or structure of the isogeny class of $E_t$. For instance, depending on the arithmetic of $E_t$, one could obtain an isogeny-torsion graph such as $L_2(2)$, or a significantly larger graph resulting from additional rational prime-degree isogenies. The goal of the remainder of this paper is to systematically distinguish between all of these possible cases. The first step in this classification is to determine whether $E_t$ admits a rational 3-isogeny and whether $E_t$ admits a rational 5-isogeny. If the answer is ``No” to both, then there are no additional rational prime-degree isogenies, and the isogeny class is of type $L_2(2)$. If the answer is ``Yes” to one, then it must be ``No” to the other, as these two conditions are mutually exclusive for non-CM elliptic curves with a rational 2-isogeny. In Proposition \ref{rational3isog} and Proposition \ref{rational5isog}, we show how to answer these two questions.

\begin{prop}\label{rational3isog}
    Let $j(t) = \frac{(256-t)^3}{t^2}$ and $p_3(u) = u^4 + 36u^3 + 270u^2 + (756 - j(t))u + 729$. If in addition to the assumptions stipulated in Theorem \ref{L_2-2max}, we also have that $p_3(u)$ has no rational root, then $E_t$ will have no rational 3-isogeny.
\end{prop}

\begin{proof}
    A rational 3-isogeny would give a non-cuspidal rational point on $X_0(3)$, hence a parameter $u \in \Q$ satisfying 
    \begin{equation*}
        j_3(u)= \frac{(u+3)^3(u+27)}{u} = j(E_t),
    \end{equation*} 
    which is equivalent to $p_3(u)=0$. Thus, if $p_3$ has no rational root, $E_t$ has no rational 3-isogeny.
\end{proof}

\begin{prop}\label{rational5isog}
    Let $j(t) = \frac{(256-t)^3}{t^2}$ and $p_5(v) = (v^2 + 10v + 5)^3 - j(t)v$. If in addition to the assumptions stipulated in Theorem \ref{L_2-2max}, we also have that $p_5(v)$ has no rational root, then $E_t$ will have no rational 5-isogeny.
\end{prop}

\begin{proof}
    A rational 5-isogeny would give a non-cuspidal rational point on $X_0(5)$, hence a parameter $v \in \Q$ satisfying 
    \begin{equation*}
        j_5(v)= \frac{(v^2+10v+5)^3}{v} = j(E_t),
    \end{equation*} 
    which is equivalent to $p_5(v)=0$. Thus, if $p_5$ has no rational root, $E_t$ has no rational 5-isogeny.
\end{proof}

\noindent Propositions \ref{rational3isog} and \ref{rational5isog} are easy to check for specific values of $t$ using Magma. We note that if $p_3(u) = 0$ has a rational solution, then $E_t$ may sit inside an isogeny class of type $R_4(6)$ or $R_6$. Similarly, if $p_5(v) = 0$ has rational solutions, then we obtain an isogeny class of size 4 with possible torsion graphs $([2],[2],[2],[2])$ or $([10],[10],[2],[2])$. From this point on, we will assume our elliptic curves have a rational 2-isogeny, and the aim is to find models for curves that also have a rational $\ell$-isogeny for $\ell \in \{3,5\}$.

\section{Elliptic Curves with a Rational 3-Isogeny}

\noindent In the previous section we described a model for an elliptic curve defined over $\Q$ with a rational 2-torsion point. We also gave criteria for when that curve was 2-maximal, and had no other rational $\ell$-isogenies for $\ell \in \{3,5 \}$. Here we start with the model for a curve with a rational 2-isogeny and rational 3-isogeny. Note that having a rational 3-isogeny does not force $E$ to have a rational 3-torsion point, so we will need to have different models for both cases. From the classification given in \cite{Chiloyan2adic}, we know that these curves can arise in four different situations:
\begin{itemize}
    \item $E$ has a rational 6-isogeny and does not have a rational 3-torsion point,
    \item $E$ has a rational 6-isogeny and a rational 3-torsion point,
    \item $E$ has a rational 9-isogeny,
    \item $E$ has two independent rational 3-isogenies.
\end{itemize}

\noindent Note that the first two cases are completely disjoint, but it is possible, for example, that an elliptic curve with a rational 6-isogeny also has a rational 9-isogeny. We will give one-parameter models for the first two cases, and then present what extra conditions need to be true to be in the third and fourth cases.

\subsection{Elliptic Curves with a Rational 6-Isogeny}

\begin{prop}\label{X_0(6)curve}
    Let $t \in \Q \setminus \{-4, -1, 0, 8\}$. Let $E_t$ be the elliptic curve defined by 
    \begin{align*}
        E_t \colon y^2 + xy = x^3 - \frac{36 t^2(t-8)^6(t+1)^3}{B(t)^2}x - \frac{t^2(t-8)^6(t+1)^3}{B(t)^2},
    \end{align*}
    where $B(t) = t^6 - 516t^5 - 12072t^4 - 24640t^3 - 30720t^2 + 6144t + 4096$.  Then $E_t$ has a rational 6-isogeny.
\end{prop}

\begin{proof}
    The curve $E_t$ comes from the $j$-map associated to $X_0(6)$.
\end{proof}

It is important to note that not every elliptic curve over $\Q$ with a rational 6-isogeny is isomorphic over $\Q$ to a curve $E_t$ in this family where $t \in \Q$. It is true, however, that there will exist some $t$ such that the $j$-invariants are equal, and thus such a statement would only hold true ``up to twists". We now proceed to find conditions on the parameter $t$ so that the curve $E_t$ is 2-maximal.

\begin{theorem}\label{X_0(6)-2max}
    Let $E_t$ be as in Proposition \ref{X_0(6)curve} above. If for all $d \in \{\pm 1, \pm2 \}$ and $f(t) \in \{t+1, t(t-8), t(t-8)(t+1) \}$, the rational number $d \cdot f(t)$ is not a perfect square in $\Q$, then $E_t$ is 2-maximal.
\end{theorem}

\begin{proof}
    The $j$-invariant of $E_t$ is given by:
    \begin{align*}
        j(E_t) = \frac{B(t)^2 + 1728t^2(t-8)^6(t+1)^3}{t^2(t-8)^6(t+1)^3},
    \end{align*}
    and since we know that this curve has a rational 2-isogeny, there must exist some function $s(t)$ satisfying $\frac{(256-s(t))^3}{s(t)^2} = j(E_t)$. After some calculations via SageMath, we find $s(t) = \frac{-t(t-8)^3}{(t+1)^3}$. Now, we proceed as in the proof of Theorem \ref{L_2-2max}, we need $d \cdot f(t)$ where $d \in \{\pm1, \pm 2\}$ and $f(t) \in \{s(t), s(t)-64, s(t)(s(t)-64)\}$ to be nonsquares in $\Q$ for rational $t$. After removing the square terms in $s(t) = \frac{-t(t-8)^3}{(t+1)^3}$, the square-free condition comes down to whether $-t(t-8)(t+1)$ is a square or not, which is what appears in our theorem statement. A similar calculation with $s(t)-64$ and $s(t)(s(t)-64)$ gives the other two polynomials in $t$.
\end{proof}

\begin{example}
    Let $t = \frac{2}{7}$. One can verify that $t$ satisfies the conditions stipulated in Theorem \ref{X_0(6)-2max}. Thus, the elliptic curve
    \begin{align*}
        E_\frac{2}{7} \colon y^2 + xy = x^3 - \frac{45927}{252004}x - \frac{5103}{1008016},
    \end{align*}
    is 2-maximal, has rational torsion subgroup isomorphic to $\Z/2\Z$ and belongs to an isogeny class of type $R_4(6)$.
\end{example}

\noindent By Proposition \ref{X_0(6)curve} and Theorem \ref{X_0(6)-2max}, these conditions guarantee when $E_t$ will have a rational 2-isogeny and a rational 3-isogeny, and be 2-maximal. By Chiloyan's classification \cite{Chiloyan2adic}, this can only happen when $E_t$ is in an isogeny class of size 4 or size 6. The following propositions determine what extra conditions need to be satisfied in order for $E_t$ to sit in an isogeny class of size 6, as well as what position it occupies on the graph.

\begin{prop}\label{X_0(6) and rational9isog}
    Let $E_t$ be as in Proposition \ref{X_0(6)curve} above. If in addition to the conditions stipulated in Theorem \ref{X_0(6)-2max}, we also have that
    \begin{equation*}
        (s^3 + 9s^2 + 27s + 3)^3 (s^3 + 9s^2 + 27s + 27) - j(E_t) \cdot s(s^2 + 9s + 27) = 0
    \end{equation*}
    has a rational solution, then $E_t$ has a rational 9-isogeny and sits in one of the corners of an isogeny graph of size 6. Moreover, if a 3-isogenous curve to $E_t$ has a rational 3-torsion point, then the torsion graph is $([6],[6],[6],[6],[2],[2])$. Otherwise, the torsion graph is $([2],[2],[2],[2],[2],[2])$.
\end{prop}

\begin{proof}
    This follows directly by equating the $j$-invariant of $E_t$ with the $j$-map for $X_0(9)$ and asking for rational solutions. The degree-12 parameterization of $X_0(9)$ factors through $X_0(3)$ as $j(s) = j(u(s))$, where the map $X_0(9) \to X_0(3)$ is given by $u(s) = s(s^2+9s+27)$, and the map $X_0(3) \to X(1)$ is given by $j(u) = \frac{(u+3)^3(u+27)}{u}$. Substituting $u(s)$ into $j(u)$ and setting it equal to $j(E_t)$ gives the desired equation after clearing denominators. The existence of a rational root $s$ confirms $E_t$ possesses a rational 9-isogeny. Because $E_t$ also possesses a rational 2-isogeny, this 9-isogeny forces the isogeny graph containing this curve to have size 6 and for $E_t$ to sit in a corner of it. (See the picture in Theorem \ref{chiloyan} cases (6) and (7)). Via the classification given by \cite{chiloyanalvaro}, the torsion graph structure then follows depending on whether a 3-isogenous neighbor $E'$ contains a rational 3-torsion point or not.
\end{proof} 

\begin{example}
    Let $t = -\frac{56}{65}$. One can verify that this value of $t$ satisfies the conditions in Theorem \ref{X_0(6)-2max} and Proposition \ref{X_0(6) and rational9isog}. We construct the elliptic curve
    \begin{align*}
        E_{-\frac{56}{65}} \colon y^2 + xy = x^3 - \frac{30057431040}{187765439779009}x -
\frac{834928640}{187765439779009}.
    \end{align*}
    This curve is 2-maximal, has rational torsion subgroup isomorphic to $\Z/2\Z$, and belongs to an isogeny class of type $R_6$. Moreover, the curve possesses a rational 9-isogeny, and thus sits in one of the corners of its isogeny graph.
\end{example}

\begin{prop}\label{X_0(6) and two independent 3isog}
    Let $E_t$ be as in Proposition \ref{X_0(6)curve} above. Let the polynomials $c_0(t), c_1(t),$ and $c_2(t)$ be defined as follows:
    {\smaller \begin{align*}
        c_0(t) =& -\frac{t^2(t - 8)^4(t + 1)^2(t^9 + 198t^8 + 170676t^7 + 2196264t^6 + 8217792t^5 + 11388672t^4 + 1339392t^3 + 1253376t^2 + 589824t + 262144)}{B(t)^3}, \\
        c_1(t) =& -\frac{9t^2(t - 8)^2(t + 1)(t^6 - 840t^5 - 7212t^4 - 40192t^3 - 51456t^2 + 6144t + 4096)}{B(t)^2}, \\
        c_2(t) =& -\frac{27t^2(7t^3 + 138t^2 + 336t + 448)}{B(t)}.
    \end{align*}}
    If $P_3(x) = x^3 + c_2(t)x^2 + c_1(t)x + c_0(t)$ has a rational root, then $E_t$ has two independent rational 3-isogenies and sits in the middle of an isogeny graph of size 6. Moreover, if $E_t$ has a rational 3-torsion point, then the torsion graph is $([6],[6],[6],[6],[2],[2])$. Otherwise, the torsion graph is $([2],[2],[2],[2],[2],[2])$.
\end{prop}

\begin{proof}
    By construction via the $j$-map associated to $X_0(6)$, the elliptic curve $E_t$ possesses both a rational 2-isogeny and a rational 3-isogeny. The existence of a rational 3-isogeny guarantees that the 3-division polynomial of $E_t$, denoted $\psi_3(x)$, factors over $\mathbb{Q}(t)$ into a linear term and a cubic factor, $P_3(x)$. We can explicitly compute the 3-division polynomial for $E_t$ using Magma:
    \begin{align*}
            \psi_3(x) = 3x^4 + x^3 - 216 C(t) x^2 - 12 C(t) x - C(t) - 1296 C(t)^2,
    \end{align*}
    where $C(t) = \frac{t^2(t-8)^6(t+1)^3}{B(t)^2}$. Factoring this polynomial over $\mathbb{Q}(t)$ gives us the cubic $P_3(x)$ as defined in the statement. Now, assume $P_3(x)$ has a rational root. This provides a second, distinct rational $x$-coordinate for a 3-torsion point on $E_t$ that gives a cyclic subgroup of order 3, say $C_2$. Because the $x$-coordinate of its generator is rational, the subgroup $C_2$ is invariant under the action of the absolute Galois group $G_{\Q}$. By \cite[Prop. III.4.12]{silverman2009arithmetic}, a finite Galois-invariant subgroup corresponds to an isogeny defined over the base field. Thus, $E_t$ has an additional, independent rational 3-isogeny. Since $E_t$ has a rational 2-isogeny and two independent rational 3-isogenies, the classification of isogeny-torsion graphs in \cite{chiloyanalvaro} gives us that the isogeny class has size 6 and $E_t$ sits in the middle of the corresponding isogeny graph. Moreover, the same classification tells us that depending on whether $E_t$ has a rational 3-torsion point, the isogeny-torsion graphs will either be $([6],[6],[6],[6],[2],[2])$ or $([2],[2],[2],[2],[2],[2])$.
\end{proof}

\begin{example}
    Let $t = \frac{8}{27}$. One can check that $t$ satisfies all the conditions required by Theorem \ref{X_0(6)-2max} and Proposition \ref{X_0(6) and two independent 3isog}. We construct the elliptic curve
    \begin{align*}
        E_\frac{8}{27} \colon y^2 + xy = x^3 -
\frac{12873910506912000}{55215681703109929}x - \frac{357608625192000}{55215681703109929}.
    \end{align*}
    This curve is 2-maximal, has rational torsion subgroup isomorphic to $\Z/2\Z$ and sits in the middle column of an isogeny graph of type $R_6$.
\end{example}

\subsection{Elliptic Curves with a Rational 3-Torsion Point}

\begin{prop}\label{X_1(6) curve}
    Let $t \in \Q \setminus \{0,-1,-\frac{1}{9}\}$ and define the elliptic curve
    \begin{align*}
        E_t \colon y^2 + (1-t)xy - (t^2+t)y = x^3 - (t^2+t)x^2.
    \end{align*}
    Then $\Z/6\Z \hookrightarrow E_t(\Q)_{\text{tor}}$. 
\end{prop}

\begin{proof}
    The point $(0,0)$ on $E_t$ is a point of order 6. 
\end{proof}

\begin{remark}
    This is the standard Tate normal form parameterizing elliptic curves equipped with a rational point of order 6 that can be found, for example, in \cite{largetorsion}. The restriction on the values of $t$ ensures that $E_t$ is nonsingular.
\end{remark}

\noindent We now determine the necessary conditions on $t$ so that $E_t$ is 2-maximal.

\begin{theorem}\label{X_1(6) 2max}
    Let $E_t$ be as in Proposition \ref{X_1(6) curve} above. If for all $d \in \{\pm 1, \pm 2\}$ and $f(t) \in \{ t, (t+1)(9t+1), t(t+1)(9t+1)\}$, the rational number $d \cdot f(t)$ is not a square in $\Q$, then $E_t$ is 2-maximal.
\end{theorem}

\begin{proof}
    The $j$-invariant of $E_t$ is given by:
    \begin{align*}
        j(E_t) = \frac{(9t^4+12t^3+30t^2+12t+1)^3}{t^6(t+1)^3(9t+1)},
    \end{align*}
    and since this curve has a rational 2-isogeny, there exists some function $s(t)$ such that we have $\frac{(256-s(t))^3}{s(t)^2} = j(E_t)$. Using SageMath, we find that $s(t) = \frac{4096t^3}{(t+1)^3(9t+1)}$. Now we simply proceed again as in Theorem \ref{L_2-2max}. Removing the square terms in $s(t)$ leaves us with $t(t+1)(9t+1)$, which is one of the polynomials stipulated in our theorem. We get the other two polynomials  from a similar calculation using $s(t) - 64$ and $s(t)(s(t)-64)$.
\end{proof}

\begin{example}
    Let $t = \frac{17}{55}$. One can verify that $t$ satisfies the conditions required by Theorem \ref{X_1(6) 2max}. We define the elliptic curve
    \begin{align*}
        E_\frac{17}{55} \colon y^2 + \frac{38}{55}xy - \frac{1224}{3025}y = x^3 - \frac{1224}{3025}x^2.
    \end{align*}
    This curve is 2-maximal, has rational torsion subgroup isomorphic to $\Z/6\Z$, and belongs to an isogeny class of type $R_4(6)$.
\end{example}

\noindent As per the classification in \cite{Chiloyan2adic}, we know that a 2-maximal curve with a rational 2-isogeny and a rational 3-isogeny can either be in an isogeny class of type $R_4(6)$ or $R_6$. The following propositions determine what other conditions we require on $t$ so that the curve sits in an isogeny graph of type $R_6$, and its location on the corresponding isogeny graph.

\begin{prop}\label{X_1(6) and rational 9isog}
    Let $E_t$ be as in Proposition \ref{X_1(6) curve} above. If we have that 
    \begin{equation*}
        (s^3 + 9s^2 + 27s + 3)^3 (s^3 + 9s^2 + 27s + 27) - j(E_t) \cdot s(s^2 + 9s + 27) = 0
    \end{equation*}
    has a rational solution, then $E_t$ has a rational 9-isogeny and sits in one of the corners of an isogeny graph of size 6. Moreover, the torsion graph is $([6],[6],[6],[6],[2],[2])$.
\end{prop}

\begin{proof}
    This argument is identical to that in Proposition \ref{X_0(6) and rational9isog}.
\end{proof} 

\begin{example}
    Let $t = \frac{1}{124}$. One can verify that this value of $t$ satisfies the conditions stipulated in Theorem \ref{X_1(6) 2max} and Proposition \ref{X_1(6) and rational 9isog}. We define the elliptic curve
    \begin{align*}
        E_\frac{1}{124} \colon y^2 + \frac{123}{124}xy - \frac{125}{15376}y = x^3 -
    \frac{125}{15376}x^2.
    \end{align*}
    This curve is 2-maximal, has rational torsion subgroup isomorphic to $\Z/6\Z$ and sits in one of the corners of an isogeny graph of type $R_6$.
\end{example}

\begin{prop}\label{X_1(6) and two independent 3isog}
    Let $E_t$ be as in Proposition \ref{X_1(6) curve} above. If $P_3(x) = x^2(x - t) + \frac{1}{3}(x - t^2 - t)^2$ has a rational root, then $E_t$ has two independent rational 3-isogenies and sits in the middle of an isogeny graph of size 6. Moreover, the torsion graph is $([6],[6],[6],[6],[2],[2])$. 
\end{prop}

\begin{proof}
    The proof follows the same framework as the $X_0(6)$ case in Proposition \ref{X_0(6) and two independent 3isog}. By construction, the universal curve $E_t$ with $\mathbb{Z}/6\mathbb{Z}$ torsion possesses both a rational 2-torsion point and a rational 3-torsion point, guaranteeing both a rational 2-isogeny and a rational 3-isogeny. The existence of the rational 3-torsion point means the 3-division polynomial $\psi_3(x)$ factors over $\Q(t)$ into a linear term and a cubic factor, $P_3(x)$. Using Magma, we compute
    \begin{align*}
        \psi_3(x) = 3x^4 + (1 - 6t - 3t^2)x^3 + 3t(t^2 - 1)x^2 + 3t^2(t+1)^2x - t^3(t+1)^3,
    \end{align*}
    and factoring this polynomial over $\Q(t)$ gives us $P_3(x)$ from the statement. If $P_3(x)$ has a rational root, then $E_t$ will have another cyclic subgroup of order 3, say $C_2$. Because the $x$-coordinate of its generator is rational, $C_2$ is invariant under the action of the absolute Galois group $G_{\Q}$. By \cite[Prop. III.4.12]{silverman2009arithmetic}, this subgroup corresponds to an independent rational 3-isogeny. Since $E_t$ possesses a rational 2-isogeny and two independent rational 3-isogenies, the classification of isogeny-torsion graphs in \cite{chiloyanalvaro} dictates that $E_t$ sits in the middle column of an isogeny graph with 6 vertices. Furthermore, because $E_t(\Q)_{\text{tor}} \cong \mathbb{Z}/6\mathbb{Z}$, the same classification restricts the isogeny-torsion graph strictly to $([6],[6],[6],[6],[2],[2])$.
\end{proof}

\begin{example}
    Let $t = \frac{19}{72}$. One can verify that this value of $t$ satisfies the conditions stipulated in Theorem \ref{X_1(6) 2max} and Proposition \ref{X_1(6) and two independent 3isog}. We define the elliptic curve
    \begin{align*}
        E_\frac{19}{72} \colon y^2 + \frac{53}{72}xy - \frac{1729}{5184}y = x^3 -
    \frac{1729}{5184}x^2.
    \end{align*}
    This curve is 2-maximal, has rational torsion subgroup isomorphic to $\Z/6\Z$, and sits in the middle column on an isogeny graph of type $R_6$.
\end{example}

\noindent This handles the case where the elliptic curve has torsion subgroup $\Z/6\Z$ and belongs to an isogeny class of size 4 or 6 and each curve in the class has 2-adic image [2.3.0.1].

\section{Elliptic Curves with a Rational 5-Isogeny}

\noindent In this section we will give models for elliptic curves with a rational 2-isogeny, and a rational 5-isogeny. From the classification in \cite{Chiloyan2adic}, curves that are 2-maximal and have these isogenies will be in an isogeny-graph of type $R_4(10)$. The only distinction we need to make is whether $E$ has a rational 5-torsion point or not. We begin here by giving the model for an elliptic curve defined over $\Q$ such that it has a rational 2-isogeny and rational 5-isogeny.

\subsection{Elliptic Curves with a Rational 10-Isogeny}

\begin{prop}\label{X_0(10) curve}
    Let $t \in \Q \setminus \{-4,0,1\}$. Let $E_t$ be the elliptic curve defined by the equation
    \begin{align*}
        E_t \colon y^2 + xy = x^3 + \frac{36A(t)}{D(t)}x + \frac{A(t)}{D(t)},
    \end{align*}
    where $A(t) = t^2(t - 1)^5(t + 4)^{10}$ and $D(t) = (t^2 + 4)(t^2 - 22t - 4)^2(t^2 - 4t + 8)^2(t^4 + 536t^3 - 264t^2 + 416t - 64)^2$. Then $E_t$ has a rational 10-isogeny.
\end{prop}

\begin{proof}
    The curve comes from the $j$-map associated to the modular curve $X_0(10)$.
\end{proof}

\noindent Now we proceed to find conditions on $t$ so that this curve is 2-maximal.

\begin{theorem}\label{X_0(10) 2max}
    Let $E_t$ be as in Proposition \ref{X_0(10) curve}. Let the polynomials $P_1(t), P_2(t), P_3(t)$ be defined as:
    \begin{align*}
        P_1(t) &= t(t+4)(t-1) \\
        P_2(t) &= (t-1)\left(64(t-1)^5 - t(t+4)^5\right) \\
        P_3(t) &= P_1(t)P_2(t).
    \end{align*}
    If for all $d \in \{\pm 1, \pm 2\}$ and $f(t) \in \{P_1(t), P_2(t), P_3(t)\}$, the rational number $d \cdot f(t)$ is not a square in $\Q$, then $E_t$ is 2-maximal.
\end{theorem}

\begin{proof}
    The $j$-invariant for $E_t$ is given by:
    \begin{align*}
        j(E_t) = 1728 - \frac{D(t)}{A(t)},
    \end{align*}
    and again, since this curve has a rational 2-isogeny, there exists some $s(t)$ such that $\frac{(256-s(t))^3}{s(t)^2} = j(E_t)$. Using SageMath, we find that $s(t) = \frac{t(t+4)^5}{(t-1)^5}$. Now we simply proceed as in Theorem \ref{L_2-2max}. After removing the square terms in $s(t)$ we are left with $s(t)$ is a square in $\Q$ if and only if $t(t+4)(t-1)$ is a square, which is exactly $P_1(t)$. We get $P_2(t)$ and $P_3(t)$ from similar calculations using $s(t) - 64$ and $s(t)(s(t)-64)$.
\end{proof}

\begin{example}
    Let $t = \frac{8}{3}$. One can verify that $t$ satisfies the conditions required in Theorem \ref{X_0(10) 2max}. We construct the elliptic curve 
    \begin{align*}
        E_\frac{8}{3} \colon y^2 + xy = x^3 + \frac{432}{45125}x + \frac{12}{45125}.
    \end{align*}
    This curve is 2-maximal, has rational torsion subgroup isomorphic to $\Z/2\Z$, and belongs to an isogeny class of type $R_4(10)$.
\end{example}

\subsection{Elliptic Curves with a Rational 5-Torsion Point}

\begin{prop} \label{X_1(10) curve}
    Let $t \in \Q \setminus \{0,1,\frac{1}{2}\}$ and define the elliptic curve
    \begin{align*}
        E_t \colon y^2 + \left(1-\frac{-2t^3 + 3t^2 - t}{t^2 - 3t + 1}\right)xy - \left(\frac{2t^5 - 3t^4 + t^3}{(t^2 -3t+1)^2} \right)y = x^3 - \left( \frac{2t^5 - 3t^4 + t^3}{(t^2 -3t+1)^2} \right)x^2.
    \end{align*}
    Then $E_t(\Q)_{\text{tor}} \cong \Z/10\Z$.
\end{prop}

\begin{proof}
    The point $(0,0)$ on $E_t$ has order 10. Mazur's torsion theorem then implies that $E_t(\Q)_{\text{tor}} \cong \Z/10\Z$.
\end{proof}

\noindent We determine conditions on $t$ that imply $E_t$ is 2-maximal. 

\begin{theorem}\label{X_1(10) 2max}
    Let $E_t$ be as in Proposition \ref{X_1(10) curve} above. Let the polynomials $Q(t), P_1(t), P_2(t), P_3(t)$ be defined as:
    \begin{align*}
        Q(t) &= (2t-1)^5(4t^2-2t-1) + 64t^5(t-1)^5(t^2-3t+1) \\
        P_1(t) &= -t(t-1)(t^2-3t+1)(2t-1)(4t^2-2t-1) \\
        P_2(t) &= Q(t)(2t-1)(4t^2-2t-1) \\
        P_3(t) &= -t(t-1)(t^2-3t+1)Q(t).
    \end{align*}
    If for all $d \in \{ \pm 1, \pm 2\}$ and $f(t) \in \{P_1(t),P_2(t),P_3(t)\}$, the rational number $d \cdot f(t)$ is not a square in $\Q$, then $E_t$ is 2-maximal.
\end{theorem}

\begin{proof}
    The $j$-invariant of $E_t$ is given by:
    \begin{align*}
        j(E_t) = \frac{32\left(t^{12}-8t^{11}+26t^{10}-45t^9+45t^8-18t^7-16t^6+27t^5-15t^4+\frac{5}{2}t^3 + t62 - \frac{1}{2}t + \frac{1}{16}\right)}{t^{10}\left(t-\frac{1}{2}\right)^5(t-1)^{10}\left(t^2-\frac{1}{2}t-\frac{1}{4}\right)(t^2-3t+1)^2},
    \end{align*}
    and since $E_t$ has a rational 2-isogeny, there exists some $s(t)$ such that $\frac{(256-s(t))^3}{s(t)^2} = j(E_t)$. Using SageMath, we get $s(t) = \frac{-4096t^5(t-1)^5(t^2-3t+1)}{(2t-1)^5(4t^2-2t-1)}$. Now we proceed as in the proof of Theorem \ref{L_2-2max}. After removing the square factors from $s(t)$, we find that $s(t)$ is a square if and only if $-t(t-1)(t^2-3t+1)(2t-1)(4t^2-2t-1)$ is a square, which is precisely $P_1(t)$. We get $P_2(t)$ and $P_3(t)$ from similar calculations using $s(t)-64$ and $s(t)(s(t)-64)$. 
\end{proof}

\begin{example}
    Let $t = \frac{19}{69}$. One can verify that $t$ satisfies the conditions required in Theorem \ref{X_1(10) 2max}. We construct the elliptic curve defined by
    \begin{align*}
        E_\frac{19}{69} \colon y^2 + \frac{111491}{82041}xy - \frac{10631450}{97546749}y = x^3 - \frac{10631450}{97546749}x^2.
    \end{align*}
    This curve is 2-maximal, has rational torsion subgroup isomorphic to $\Z/10\Z$, and sits in an isogeny graph of type $R_4(10)$.
\end{example}

\noindent \textbf{Acknowledgements.} The author would like to express his gratitude to Nathan Kaplan for feedback on an earlier draft and for many fruitful conversations in the preparation of this manuscript. The author would also like to thank Alex Barrios and David Zureick-Brown for comments on an earlier version of this manuscript.

\clearpage
\bibliographystyle{plain}
\bibliography{references}

\end{document}